\documentclass[10pt, a4paper]{amsart}

\usepackage[utf8]{inputenc}

\usepackage[english]{babel}

\usepackage{microtype}
\usepackage{amsmath}
\usepackage{amsthm}
\usepackage{amssymb}

\newtheorem{theorem}{Theorem}
\newtheorem{lemma}{Lemma}
\newtheorem{proposition}{Proposition}

\usepackage[backend=biber, bibstyle=math-numeric-normalsize, citestyle=math-numeric]{biblatex}

\usepackage[hidelinks]{hyperref}
\hypersetup{
	pdftitle={Almost-palindromic width of free groups},
    pdfauthor={M. Staiger}
}

\DeclareMathOperator{\len}{length}
\DeclareMathOperator{\wid}{width}
\DeclareMathOperator{\apwid}{AP-width}
\DeclareMathOperator{\rev}{rev}
\DeclareMathOperator{\sign}{sign}

\begin{document}

\title[Almost-palindromic width of free groups]{On the almost-palindromic width of free groups}
\author{Manuel Staiger}
\address{Freie Universität\\Institut für Mathematik\\Arnimallee 2\\14195 Berlin}
\email{manuel.staiger@fu-berlin.de}
\date{June 27, 2023. \textit{Revised}:~July 4, 2024}

\begin{abstract}
We answer a question of Bardakov (Kourovka Notebook, Problem~19.8) which asks for the existence of a pair of natural numbers $(c, m)$ with the property that every element in the free group on the two-element set $\{a, b\}$ can be represented as a concatenation of $c$, or fewer, $m$-almost-palindromes in letters $a^{\pm 1}, b^{\pm 1}$. Here, an $m$-almost-palindrome is a word which can be obtained from a palindrome by changing at most $m$ letters. We show that no such pair $(c, m)$ exists. In fact, we show that the analogous result holds for all non-abelian free groups.
\end{abstract}

\maketitle

\pagestyle{myheadings}
\markboth{}{}

\section{Introduction and statement of the main result}

For a group $G$, an element $g\in G$, and a generating set $X\subseteq G$, let $\len(g, X)$ be the minimal $n\geq 0$ with the property that there exist $x_1, \dots, x_n \in X$ for which $g = x_1^{\pm 1}\dotsm x_n^{\pm 1}$, that is, the minimal number of elements from the generating set $X$ necessary to generate $g$. Let us define the width of $G$ with respect to the generating set $X$ as
$$\wid(G, X) = \max_{g\in G} \len(g, X) \,,$$
or $\wid(G, X) = \infty$, if the maximum does not exist.

If $B$ is a set, we denote by $F_B$ the free group on $B$ whose elements are all reduced words with letters in $B^{\pm 1} = B\cup \{b^{-1}\,\vert\, b\in B\}$, with the operation given by concatenation and subsequent reduction. We denote by $W_B$ the free monoid on the set $B^{\pm 1}$ whose elements are now all words with letters in $B^{\pm 1}$, with concatenation as operation. There is the homomorphism of monoids $r_B\colon W_B\to F_B$ which sends a word to its corresponding reduced word, and which is left-inverse to the inclusion of $F_B$ into $W_B$.

For a word $w\in W_B$, let $\rev(w)$ be its reverse word, that is, the word given by the letters of $w$ in reverse order. A palindrome is a word $p\in W_B$ with the property that $p = \rev(p)$. An $m$-almost-palindrome is a word which differs from a palindrome by a change of at most $m$ letters. In other words, $\tilde{p}\in W_B$ is an $m$-almost-palindrome if there exists a palindrome $p$ with the property that $d(p, \tilde{p})\leq m$, where $d$ is the Hamming distance on the set of all words in $W_B$ whose number of letters equals the number of letters of $\tilde{p}$, or, equivalently, if $d(\tilde{p}, \rev(\tilde{p}))\leq 2m$. Thus, the palindromes are exactly the $0$-almost-palindromes. Let us denote by $P_{B, m}$ the set of all $m$-almost-palindromes and observe that the image $r_B(P_{B, m})$ is a generating set for the free group $F_B$, for every $m\geq 0$.

If $F$ is a free group, together with a basis $B$ of $F$, then there is the canonical isomorphism $\phi_B\colon F_B\to F$. Furthermore, note that for all $m$ the set $X_{B, m} = \phi_B(r_B(P_{B, m}))$ is a generating set for $F$. In fact, we have an increasing sequence $X_{B,0}\subseteq X_{B,1}\subseteq\dots$ of generating sets of $F$. The value of the expression $\wid(F, X_{B, m})$ does not depend on the choice of the basis $B$ of $F$, since for two bases $B, B'$, the automorphism of $F$ induced by a bijection $B\to B'$ sends $X_{B, m}$ to $X_{B', m}$. We call this value the $m$-almost-palindromic width of the free group $F$ and write $\apwid(F, m) = \wid(F, X_{B, m})$.

If $F$ is trivial, or a free group of rank one, then it is readily seen that for all $m\geq 0$ the $m$-almost-palindromic width of $F$ is zero, or one. However, \citeauthor{Bardakov2005}~\cite{Bardakov2005} show that, for a non-abelian free group $F$ (i.e., a free group which is neither trivial nor infinite cyclic), the $0$-almost-palindromic width (that is, the palindromic width) of $F$ is infinite. Extending on this, Bardakov \cite[Problem~19.8]{Khukhro2018} asks whether there exists a pair of natural numbers $(c, m)$ with the property that every element of $F_{\{a, b\}}$ can be represented by a concatenation of $c$, or fewer, $m$-almost-palindromes in letters $a^{\pm 1}, b^{\pm 1}$. In other words, he asks whether there is an $m\in\mathbb{N}$ for which the $m$-almost-palindromic width of $F$ is finite, assuming that $F$ is a free group of rank two.

Of course, the same question can be asked without the restriction on the rank of~$F$. In this note, we answer this question negatively for all non-abelian free groups, building on the argument by \citeauthor{Bardakov2005} for the case $m=0$.

\begin{theorem}
\label{thm:maintheorem}
If $F$ is a non-abelian free group, then, for all $m\geq 0$, the $m$-almost-palindromic width of $F$ is infinite.
\end{theorem}

Throughout this note, whenever $u, w\in W_B$, in writing $uw$ we always mean the concatenation of the words $u$ and $w$, that is, the result under the operation in the monoid $W_B$, in order to avoid ambiguity which may arise from the fact that $F_B\subseteq W_B$. For instance, we explicitly write $r_B(uw)$ to denote the result under the operation in the group $F_B$, if $u$ and $w$ are reduced words. Moreover, if $u, w\in W_B$ are words, we mean by $u = w$ that the two are equal as elements of $W_B$, i.e., in a letter-by-letter way.
We denote the empty word by $\varepsilon$. If $t\in B^{\pm 1}$ and $k\in \mathbb{N}$, then we denote by $t^k$ the word which consists of $k$-times the letter $t$. In particular, $t^0 = \varepsilon$, while $t^1$ denotes the word which consists of the single letter $t$. If $t\in B$ and $k<0$, then we write $t^k=(t^{-1})^{-k}$, that is, $(-k)$-times the letter $t^{-1}\in B^{\pm 1}$.

\section{A map with a useful property}

If $B$ is a set, and $w\in F_B$, then, for $n\geq 0$, we may write $w = t_1^{k_1}\dotsm t_n^{k_n}$, where $t_i\in B$ and $k_i\in \mathbb{Z}\setminus\{0\}$ for all $i\in\{1,\dots,n\}$, and where $t_i\neq t_{i+1}$ for all $i\in\{1,\dots, n-1\}$. In this way, the number $n$, as well as the $k_i$ and $t_i$ are uniquely determined. If $n=0$, this means that $w = \varepsilon$. We call the words ${t_i}^{k_i}$, for $i\in\{1,\dots, n\}$, the syllables of the reduced word $w$.

Let the map $\Delta_B\colon W_ B\to \mathbb{Z}$ be defined in two steps as follows: First, if $w\in F_B$ is of the form $w = t^k$ for $t\in B, k\in\mathbb{Z}$, then we define
$$\Delta_B(w) = \Delta_B(t^k) = 0 \,.$$
In particular, $\Delta_B(\varepsilon) = 0$. Otherwise, $w\in F_B$ has at least two syllables, that is, it is of the form $w = t_1^{k_1}\dotsm t_n^{k_n}$, where the $t_i^{k_i}$ are the syllables of $w$, with $t_i\in B$ for all $i\in\{1,\dots, n\}$, for $n\geq 2$. In this case, we set
$$\Delta_B(w) = \Delta_B(t_1^{k_1}\dotsm t_n^{k_n}) = \sum_{i=1}^{n-1} \sign(\lvert k_{i+1}\rvert - \lvert k_{i} \rvert) \,.$$
Here, $\sign\colon \mathbb{Z}\to \mathbb{Z}$ is the map which takes the value $-1$ on all negative integers, the value $1$ on all positive integers, and for which $\sign(0) = 0$.
This concludes the definition of $\Delta_B$ on $F_B$. In order to extend this definition to all of $W_B$, we define
$$\Delta_B(w) = \Delta_B(r_B(w)) \,,$$
for a not necessarily reduced word $w\in W_B$.
\pagebreak

The map $\Delta_B$ has the following useful property which resembles \cite[Lemma~1.3]{Bardakov2005}.
\begin{lemma}
\label{lem:lemma}
Let $n\geq 0$ be an integer and $w_1,\dots ,w_n\in W_B$. Then
$$\left\lvert \Delta_B(w_1\dotsm w_n) - \sum_{i=1}^n \Delta_B(w_i)\right\rvert \leq 6n \,.$$
\end{lemma}
\begin{proof}
If $n = 0$, then $\Delta_B(w_1\dotsm w_n) = \Delta_B(\varepsilon) = 0$, as well as $\sum_{i=1}^n \Delta_B(w_i) = 0$. Thus, for $n=0$ the inequality holds. If $n = 1$, the inequality is obviously true.

For $n = 2$, we consider two cases: First, if $r_B(w_2) = r_B(w_1)^{-1}$, then $\Delta_B(w_1 w_2) = \Delta_B(r_B(w_1 w_2)) = \Delta_B(r_B(w_1)r_B(w_2)) = \Delta_B(\varepsilon) = 0$. Also, it is readily seen that $\Delta_B(w_1) + \Delta_B(w_2) = \Delta_B(r_B(w_1)) + \Delta_B(r_B(w_1)^{-1}) = 0$, so, in this case,
$$\lvert \Delta_B(w_1 w_2) - (\Delta_B(w_1)+\Delta_B(w_2))\rvert = 0\leq 12 = 6n \,.$$
Otherwise, let $r_B(w_1) = t_n^{k_n}\dotsm t_1^{k_1}$ and $r_B(w_2) = s_1^{l_1}\dotsm s_m^{l_m}$ be the decompositions of $r_B(w_1)$ and $r_B(w_2)$ into syllables, where $t_i\in B, k_i\in \mathbb{Z}\setminus\{0\}$ for all $i\in\{1,\dots ,n\}$, and $s_i\in B, l_i\in \mathbb{Z}\setminus\{0\}$ for all $i\in\{1,\dots ,m\}$. Let $j$ be the largest index with the property that $0\leq j\leq \min\{n, m\}$, and that ${t_1}\!\!^{-k_1}\dotsm {t_j}\!\!^{-k_j} = {s_1}\!^{l_1}\dotsm {s_j}^{l_j}$. We set $z = {s_1}\!^{l_1}\dotsm {s_j}^{l_j}$ (i.e., $z = \varepsilon$, if $j=0$) and, since $r_B(w_2)\neq r_B(w_1)^{-1}$, we may write
\begin{align*}
r_B(w_1) &= u_1 t^k z^{-1} \,, \\
r_B(w_2) &= z t^l u_2 \,,
\end{align*}
where $t\in B$ is a letter, $k, l\in \mathbb{Z}$ with $k+l\neq 0$, where neither $u_1$ ends with the letter $t$ or $t^{-1}$, nor $u_2$ begins with $t$ or $t^{-1}$, and where $u_1$ and $z^{-1}$ each consist of whole syllables of $r_B(w_1)$ (i.e., the sum of the numbers of syllables of the reduced words $u_1$, $t^k$, and $z^{-1}$ is equal to the number of syllables of $r_B(w_1)$), and where $z$ and $u_2$ each consist of whole syllables of $r_B(w_2)$. Here, it may happen that $u_1$, $u_2$, or $z$, of course, are empty. Also, one of the numbers $k$ and $l$ may be zero.

Hence, we observe that
\begin{align*}
\Delta_B(w_1) &=  \Delta_B(u_1) + \epsilon_1 + \Delta_B(z^{-1}) \,, \\
\Delta_B(w_2) &= \Delta_B(z) +\epsilon_2 + \Delta_B(u_2) \,,
\end{align*}
where $\epsilon_1, \epsilon_2\in \{-2, -1, 0, 1, 2\}$. Furthermore,
$$\Delta_B(w_1 w_2) = \Delta_B(u_1 t^k z^{-1} z t^l u_2) = \Delta_B(u_1 t^{k+l} u_2) = \Delta_B(u_1) +\epsilon_3 + \Delta_B(u_2) \,,$$
where, similarly, $\epsilon_3 \in \{-2, -1, 0, 1, 2\}$. Now,
$$\lvert\Delta_B(w_1 w_2) - (\Delta_B(w_1)+\Delta_B(w_2))\rvert = \lvert-\epsilon_1-\epsilon_2+\epsilon_3\rvert \leq 6 \leq 12 = 6n \,.$$

Finally, for $n\geq 3$, we proceed by induction: Assuming that the assertion holds for $n-1$, and applying the considerations for the case $n = 2$, we see that
\begin{align*}
\Bigg\lvert \Delta_B(w_1\dotsm w_n) - \sum_{i=1}^n \Delta_B(w_i)\Bigg\rvert &= \Bigg\lvert \Delta_B(w_1\dotsm w_n) - (\Delta_B(w_1\dotsm w_{n-1}) + \Delta_B(w_n)) \\
&\qquad + (\Delta_B(w_1\dotsm w_{n-1}) + \Delta_B(w_n)) -\sum_{i=1}^n \Delta_B(w_i) \Bigg\rvert \\
&\leq \Bigg\lvert \Delta_B(w_1\dotsm w_n) - (\Delta_B(w_1\dotsm w_{n-1}) + \Delta_B(w_n))\Bigg\rvert \\
&\qquad + \left\lvert \Delta_B(w_1\dotsm w_{n-1}) - \sum_{i=1}^{n-1} \Delta_B(w_i) \right\rvert \\
&\leq 6 + 6(n-1) \\
&= 6n \,,
\end{align*} \enlargethispage*{0.2cm}
completing the proof.
\end{proof}

\section{Bounds for some values of the map}

Throughout this section, let $B$ be a fixed set. We simplify notation and put $F = F_B$,  $W = W_B$, $P_m = P_{B, m}$, $r = r_B$, and $\Delta = \Delta_B$. In this context, the following propositions hold.

\begin{proposition}
\label{prop:single-ap-bound}
If $\tilde{p}\in P_m$ is an $m$-almost-palindrome, then
$$\Delta(\tilde{p})\leq 24m + 12 \,.$$
\end{proposition}
\begin{proof}
Let $p$ be a palindrome from which $\tilde{p}$ is obtained by changing $m$ or fewer letters. Let $t_1,\dots , t_\alpha\in B^{\pm 1}$ be all the letters of $p$ which are subject to the change, in the order in which they appear within $p$, so that
$$p = w_1 t_1^1 w_2 t_2^1\dotsm t_\alpha^1 w_{\alpha+1} \,,$$
where the $w_i$ are (potentially empty) words. Here, $\alpha\leq m$.

By Lemma~\ref{lem:lemma}, we have
$$\Delta(p) - \sum_{i=1}^{\alpha+1} \Delta(w_i) - \sum_{i=1}^\alpha \Delta(t_i^1) \geq -6\cdot (2\alpha + 1) \,.$$
Here, $p$ is a palindrome, and the $t_i^1$ are words with exactly one letter, so $\Delta(p) = 0$, as well as $\Delta(t_i^1) = 0$ for all $i\in\{1,\dots ,\alpha\}$. Hence, the latter inequality can be simplified to yield
$$\sum_{i=1}^{\alpha+1} \Delta(w_i) \leq 12\alpha + 6 \,.$$

Let $\tilde{t}_i$ be the letter $t_i$ after the relevant change, so that
$$\tilde{p} = w_1 \tilde{t}_1^1 w_2 \tilde{t}_2^1\dotsm \tilde{t}_\alpha^1 w_{\alpha+1} \,.$$
Still, we have $\Delta(\tilde{t}_i^1) = 0$, for all $i\in\{1,\dots ,\alpha\}$.

Now, using Lemma~\ref{lem:lemma}, together with the above observations, we arrive at the desired inequality
\begin{align*}
\Delta(\tilde{p}) &= \Delta(w_1 \tilde{t}_1^1 w_2 \tilde{t}_2^1\dotsm \tilde{t}_\alpha^1 w_{\alpha+1}) \\
&\leq \left(\sum_{i=1}^{\alpha+1} \Delta(w_i)\right) + \left( \sum_{i=1}^\alpha \Delta(\tilde{t}_i^1)\right) + 6\cdot (2\alpha + 1) \\
&\leq (12\alpha + 6) + 0 + (12\alpha + 6) \\
&= 24\alpha + 12 \\
&\leq 24m + 12 \,.
\end{align*}
\end{proof}

\begin{proposition}
\label{prop:concat-ap-bound}
Let $c\geq 0$ be an integer, and let $\tilde{p}_1,\dots,\tilde{p}_c\in P_m$ be $m$-almost-palindromes. Then
$$\Delta(\tilde{p}_1\dotsm\tilde{p}_c) \leq 24mc + 18c \,.$$
\end{proposition}
\begin{proof}
We use Lemma~\ref{lem:lemma} and then apply Proposition~\ref{prop:single-ap-bound} to see that
\begin{align*}
\Delta(\tilde{p}_1\dotsm\tilde{p}_c) &\leq \left(\sum_{i=1}^c \Delta(\tilde{p}_i)\right) + 6c \\
&\leq c\cdot (24m + 12) + 6c \\
&=24mc + 18c \,.
\end{align*}
\end{proof}

\section{Proof of the main result}

\begin{proof}[Proof of Theorem~\ref{thm:maintheorem}.]
Let $F$ be a non-abelian free group, and $m\geq 0$. Since $F$ is non-abelian, a basis $B$ has at least two distinct elements $a, b\in B$. Consider the sequence of words $(w_n)_{n\in\mathbb{N}}$ defined by
$$w_n = a^1 b^1 a^2 b^2\dotsm a^n b^n \in F_B \,.$$
It is readily seen that $\Delta_B(w_n) = n-1$. In particular, $\Delta_B(w_n)\to\infty$ for $n\to\infty$.

On the other hand, for $c\geq 0$, Proposition~\ref{prop:concat-ap-bound} shows that $\Delta_B$ is bounded from above on the set $r_B({P_{B, m}}^c)\subseteq F_B$, where ${P_{B, m}}^c = \{p_1\dotsm p_c \,\vert\, p_i\in P_{B, m} \text{ for all } i\in\{1,\dots, c\}\}$.
Hence, $r_B({P_{B, m}}^c)\neq F_B$ for all $c\geq 0$. This implies that for all $c\geq 0$ it holds that also
$$\phi_B(r_B({P_{B, m}}^c))=\{x_1 \dotsm x_c \,\vert\, x_i\in X_{B, m} \text{ for all } i\in \{1,\dots ,c\}\} \neq F \,.$$
Thus, $\apwid(F, m) = \infty$.
\end{proof}

\section*{Acknowledgments}

The author would like to thank Pavle Blagojević and Kıvanç Ersoy for their continuous support, as well as Tatiana Levinson and Nikola Sadovek for useful discussions. 

\printbibliography

@Misc{Khukhro2018,
  author = {Khukhro, E. I. and Mazurov, V. D.},
  title  = {Unsolved Problems in Group Theory. The Kourovka Notebook. No. 19},
  year   = {2018}
}

@Article{Bardakov2005,
  author  = {Bardakov, V. and Shpilrain, V. and Tolstykh, V.},
  title   = {On the palindromic and primitive widths of a free group},
  journal = {Journal of Algebra},
  year    = {2005},
  volume  = {285},
  number  = {2},
  pages   = {574-585}
}
\end{document}